\title{Weak and strong $k$-connectivity games}
\author{{Asaf Ferber \thanks{School of Mathematical Sciences, Raymond and Beverly Sackler Faculty of Exact Sciences, Tel Aviv University, Tel Aviv, 69978, Israel. Email: ferberas@post.tau.ac.il}}
\quad {Dan Hefetz \thanks{School of Mathematics, University of Birmingham, Edgbaston, Birmingham B15 2TT, United Kingdom. Email: d.hefetz@bham.ac.uk}}}

\documentclass[11pt]{article}
\usepackage{amsmath,amssymb,latexsym,color,epsfig,a4,enumerate}

\newif\ifnotesw\noteswtrue

\newtheorem{theorem}{Theorem}[section]
\newtheorem{lemma}[theorem]{Lemma}

\newtheorem{proposition}[theorem]{Proposition}

\newtheorem{corollary}[theorem]{Corollary}

\newtheorem{remark}[theorem]{Remark}

\newenvironment{proof}{\noindent{\bf Proof\,}}{\hfill$\Box$}

\begin{document}
\maketitle

\begin{abstract}
For a positive integer $k$ we consider the $k$-vertex-connectivity
game, played on the edge set of $K_n$, the complete graph on $n$ vertices. 
We first study the Maker-Breaker version of this game and prove that, for any
integer $k \geq 2$ and sufficiently large $n$, Maker has a strategy for
winning this game within $\lfloor k n/2 \rfloor + 1$ moves, which is clearly 
best possible. This answers a question from~\cite{HKSS}. We then consider the
strong $k$-vertex-connectivity game. For every positive integer $k$ and
sufficiently large $n$, we describe an explicit first player's
winning strategy for this game. 
\end{abstract}

\section{Introduction}
\label{sec::intro}

Let $X$ be a finite set and let ${\mathcal F} \subseteq 2^X$ be a
family of subsets. In the \emph{strong game} $(X, {\mathcal F})$,
two players, called Red and Blue, take turns in claiming one
previously unclaimed element of $X$, with Red going first. The
winner of the game is the first player to fully claim some $F \in
{\mathcal F}$. If neither player is able to fully claim some $F \in
{\mathcal F}$ by the time every element of $X$ has been claimed by
some player, the game ends in a \emph{draw}. The set $X$ will be
referred to as the \emph{board} of the game and the elements of
${\mathcal F}$ will be referred to as the \emph{winning sets}.

It is well known from classic Game Theory that, for every strong game $(X, {\mathcal F})$,
either Red has a winning strategy (that is, he is able to win the game against any
strategy of Blue) or Blue has a drawing strategy (that is, he is able to avoid losing
the game against any strategy of Red; a \emph{strategy stealing} argument shows that
Blue cannot win the game). For certain games, a hypergraph coloring argument
can be used to prove that draw is impossible and thus these games are won by Red.
However, the aforementioned arguments are purely existential. That is, even if it is known
that Red has a winning strategy for some strong game $(X, {\mathcal F})$, it might
be very hard to describe such a strategy explicitly. The few examples of natural games
for which an explicit winning strategy is known include the \emph{perfect matching}
and \emph{Hamilton cycle} games (see~\cite{FH}). 

Partly due to the great difficulty of studying strong games, weak games were 
introduced. In the \emph{Maker-Breaker game} (also known as \emph{weak} game)
$(X,{\mathcal F})$, two players, called Maker and Breaker, take
turns in claiming previously unclaimed elements of $X$, with Breaker
going first (in some cases it will be convenient to assume that Maker
starts the game; whenever this assumption is made, it will be stated explicitly). 
Each player claims \textbf{exactly} one element of $X$ per turn (sometimes 
it will be convenient to assume that each player claims \textbf{at most} 
one element of $X$ per turn; since Maker-Breaker games are bias monotone,
this has no effect on the outcome of the game). The set
$X$ is called the \emph{board} of the game and the members of
${\mathcal F}$ are referred to as the \emph{winning sets}. Maker
wins the game as soon as he occupies all elements of some winning
set. If Maker does not fully occupy any winning set by the time
every board element is claimed by some player, then Breaker wins the
game. Note that being the first player is never a disadvantage in a
Maker-Breaker game. Hence, in order to prove that Maker can win some
Maker-Breaker game as the first or second player, it suffices to
prove that he can win this game as the second player.

In this paper we study the weak and strong versions of the 
$k$-vertex-connectivity game $(E(K_n), {\mathcal C}_n^k)$. 
The board of this game is the edge set of the complete graph on $n$
vertices and its family of winning sets ${\mathcal C}_n^k$, consists
of the edge sets of all $k$-vertex-connected subgraphs of $K_n$. 

It is easy to see (and also follows from~\cite{Lehman}) that, for
every $n \geq 4$, Maker can win the weak game $(E(K_n), {\mathcal
C}_n^1)$ within $n-1$ moves. Clearly this is best possible. It
follows from~\cite{HS} that, if $n$ is not too small, then Maker can
win the weak game $(E(K_n), {\mathcal C}_n^2)$ within $n+1$ moves
and this is best possible as well. It was proved in~\cite{HKSS}
that, for every fixed $k \geq 3$ and sufficiently large $n$, Maker
can win the weak game $(E(K_n), {\mathcal C}_n^k)$ within $kn/2 + (k
+ 4)(\sqrt{n} + 2n^{2/3} \log n)$ moves. Since, clearly Maker
cannot win this game in less than $k n/2$ moves, this shows that the
number of excess moves Maker plays is $o(n)$. It was asked
in~\cite{HKSS} whether the dependency in $n$ of the number of excess
moves can be omitted, that is, whether Maker can win $(E(K_n),
{\mathcal C}_n^k)$ within $kn/2 + c_k$ moves for some $c_k$ which is
independent of $n$. We answer this question in the affirmative.

\begin{theorem} \label{th::kConMakerBreaker}
let $k \geq 2$ be an integer and let $n$ be a sufficiently large
integer. Then Maker (as the first or second player) has a strategy 
for winning the weak game $(E(K_n), {\mathcal C}_n^k)$ within at most 
$\lfloor k n/2 \rfloor + 1$ moves.
\end{theorem}

The upper bound on the number of moves obtained in
Theorem~\ref{th::kConMakerBreaker} is clearly best possible.

In the minimum-degree-$k$ game $(E(K_n), {\mathcal D}_n^k)$, the board
is again the edge set of $K_n$ and the family of winning sets
${\mathcal D}_n^k$, consists of the edge sets of all subgraphs of $K_n$
with minimum degree at least $k$. Since ${\mathcal C}_n^k \subseteq
{\mathcal D}_n^k$ for every $k$ and $n$ we immediately obtain the
following result.

\begin{corollary} \label{cor::MinDegk}
Let $k \geq 1$ be an integer and let $n$ be a sufficiently large
integer. Then Maker (as the first or second player) has a strategy 
to win the weak game $(E(K_n), {\mathcal D}_n^k)$ within at most 
$\lfloor k n/2 \rfloor + 1$ moves.
\end{corollary}

It is easy to see that Maker cannot win $(E(K_n), {\mathcal D}_n^k)$
within $\lfloor k n/2 \rfloor$ moves. Hence, the bound stated in
Corollary~\ref{cor::MinDegk} is tight.

Note that, for $k=1$, Corollary~\ref{cor::MinDegk} does not follow 
from Theorem~\ref{th::kConMakerBreaker}. However, this case was proved
in~\cite{HKSS}. Moreover, we will prove a strengthening of this result
in Section~\ref{sec::auxiliary}.

It was observed in~\cite{FH} that a fast winning strategy for Maker
in the weak game $(X, {\mathcal F})$ has the potential of being used
to devise a winning strategy for the first player in the strong game
$(X, {\mathcal F})$. Using our strategy for the weak game $(E(K_n), 
{\mathcal C}_n^k)$, we will devise an explicit winning strategy for 
the corresponding strong game. We restrict our attention to the case
$k \geq 3$ as the (much simpler) cases $k=1$ and $k=2$ were discussed
in~\cite{FH}.

\begin{theorem} \label{th::kConStrongGame}
let $k \geq 3$ be an integer and let $n$ be a sufficiently large
integer. Then Red has a strategy to win the strong game 
$(E(K_n), {\mathcal C}_n^k)$ within at most $\lfloor k n/2 \rfloor + 1$ moves.
\end{theorem}

Our proof of Theorem~\ref{th::kConStrongGame} will in fact show that Red can 
build a $k$-vertex-connected graph before Blue can build a graph with minimum degree 
at least $k$. We thus have the following corollary. 

\begin{corollary} \label{cor::StrongMinDegk}
Let $k \geq 1$ be an integer and let $n$ be a sufficiently large
integer. Then Red has a strategy to win the strong game $(E(K_n),
{\mathcal D}_n^k)$ within at most $\lfloor k n/2 \rfloor + 1$ moves.
\end{corollary}

As with Corollary~\ref{cor::MinDegk}, the cases $k=1$ and $k=2$
do not follow from Theorem~\ref{th::kConStrongGame}. However, these
simple cases were discussed in~\cite{FH}. Moreover, for $k=1$ we will 
prove a strengthening of this result in Section~\ref{sec::auxiliary}. 

\noindent The rest of this paper is organized as follows: in
Subsection~\ref{subsec::notation} we introduce some notation and
terminology that will be used throughout this paper. In
Section~\ref{sec::families} we describe a family of
$k$-vertex-connected graphs that will be used in the proofs of
Theorems~\ref{th::kConMakerBreaker} and~\ref{th::kConStrongGame}. 
In Section~\ref{sec::auxiliary} we study certain simple games; the
results obtained will be used in the following sections.
In Section~\ref{sec::kConMB} we prove
Theorem~\ref{th::kConMakerBreaker} and in
Section~\ref{sec::strongCon} we prove Theorem~\ref{th::kConStrongGame}. 
Finally, in Section~\ref{sec::openprob} we present some open problems.

\subsection{Notation and terminology}
\label{subsec::notation}

\noindent Our graph-theoretic notation is
standard and follows that of~\cite{West}. In particular, we use the
following.

For a graph $G$, let $V(G)$ and $E(G)$ denote its sets of vertices
and edges respectively, and let $v(G) = |V(G)|$ and $e(G) = |E(G)|$.
For disjoint sets $A,B \subseteq V(G)$, let $E_G(A,B)$ denote the
set of edges of $G$ with one endpoint in $A$ and one endpoint in
$B$, and let $e_G(A,B) = |E_G(A,B)|$. For a set $S \subseteq V(G)$,
let $G[S]$ denote the subgraph of $G$ which is induced on the set
$S$. For disjoint sets $S,T \subseteq V(G)$, let $N_G(S,T) = \{u 
\in T : \exists v \in S, uv \in E(G)\}$ denote the set of
neighbors of the vertices of $S$ in $T$.
For a set $T \subseteq V(G)$ and a vertex $w \in V(G) \setminus T$ 
we abbreviate $N_G(\{w\}, T)$ to $N_G(w, T)$, and let $d_G(w,T)
= |N_G(w,T)|$ denote the degree of $w$ into $T$. For a set $S \subseteq V(G)$ 
and a vertex $w \in V(G)$ we abbreviate $N_G(S, V(G) \setminus S)$ to $N_G(S)$
and $N_G(w, V(G) \setminus \{w\})$ to $N_G(w)$. We let $d_G(w)
= |N_G(w)|$ denote the degree of $w$ in $G$. The minimum and maximum
degrees of a graph $G$ are denoted by $\delta(G)$ and $\Delta(G)$
respectively. For vertices $u,v \in V(G)$ let $dist_G(u,v)$ denote the 
\emph{distance} between $u$ and $v$ in $G$, that is, the number of edges 
in a shortest path of $G$, connecting $u$ and $v$. Often, when there is no 
risk of confusion, we omit the subscript $G$ from the notation above.
For a positive integer $k$, let $[k]$ denote the set $\{1, \ldots, k\}$.

Assume that some Maker-Breaker game, played on the edge set of some
graph $G$, is in progress. At any given moment during this game, we
denote the graph spanned by Maker's edges by $M$ and the graph
spanned by Breaker's edges by $B$. At any point during the game, the
edges of $G \setminus (M \cup B)$ are called \emph{free}.

Similarly, assume that some strong game, played on the edge set of some graph
$G$, is in progress. At any given moment during this game, we denote
the graph spanned by Red's edges by $R$ and the graph spanned by
Blue's edges by $B$. At any point during the game, the edges of $G
\setminus (R \cup B)$ are called \emph{free}.

\section{A family of $k$-vertex-connected graphs}
\label{sec::families}

In this section we describe a family of $k$-vertex-connected graphs. 
We will use this family in the proofs of Theorem~\ref{th::kConMakerBreaker} 
and Theorem~\ref{th::kConStrongGame}. 

Let $k \geq 3$ be an integer and let $n$ be a sufficiently large integer.  
Let ${\mathcal G}_k$ be the family of all graphs $G_k = (V,E_k)$ on $n$ vertices 
for which there exists a partition $V = V_1 \cup \ldots \cup V_{k-1}$ such that 
all of the following properties hold:
\begin{description} 
\item [(i)] $|V_i| \geq 5$ for every $1 \leq i \leq k-1$.
\item [(ii)] $\delta(G_k) \geq k$.
\item [(iii)] $G_k[V_i]$ admits a Hamilton cycle $C_i$ for every $1 \leq i \leq k-1$.
\item [(iv)] For every $1 \leq i < j \leq k-1$ the bipartite subgraph of $G_k$
with parts $V_i$ and $V_j$ admits a matching of size 3.
\item [(v)] For every $1 \leq i \leq k-1$ and every $u \in V_i$, $|\{j \in [k-1] \setminus \{i\} 
: d_{G_k}(u, V_j) = 0\}| \leq 1$.
\item [(vi)] For every $1 \leq i \leq k-1$ and every $u, v \in V_i$,
if $|\{j \in [k-1] \setminus \{i\} : d_{G_k}(u, V_j) = 0\}| = 
|\{j \in [k-1] \setminus \{i\} : d_{G_k}(v, V_j) = 0\}| = 1$, then 
$dist_{C_i}(u,v) \geq 2$.
\end{description}

\begin{proposition} \label{prop::kconnected}
For every integer $k \geq 3$ and sufficiently large integer $n$, 
every $G_k \in {\mathcal G}_k$ is $k$-vertex-connected.
\end{proposition}

\begin{proof}
Let $G_k$ be any graph in ${\mathcal G}_k$. Let $S \subseteq V$ be
an arbitrary set of size $k-1$. We will prove that $G_k
\setminus S$ is connected. We distinguish between the following 
three cases.

\begin{description} 
\item [Case 1:] $|S \cap V_i| = 1$ for every $1 \leq i \leq k-1$.\\
Since $G_k[V_i]$ is Hamiltonian for every $1 \leq i \leq k-1$ by
Property (iii) above, it follows that $(G_k \setminus S)[V_i]$ is 
connected for every $1 \leq i \leq k-1$. Hence, in order to prove that
$G_k \setminus S$ is connected, it suffices to prove that $E_{G_k \setminus S}(V_i, V_j) 
\neq \emptyset$ holds for every $1 \leq i < j \leq k-1$. Fix
some $1 \leq i < j \leq k-1$. It follows by Property (iv) above that 
there exist vertices $x_i, y_i, z_i \in V_i$ and $x_j, y_j, z_j \in V_j$
such that $x_i x_j, y_i y_j, z_i z_j \in E_{G_k}(V_i, V_j)$. Clearly, 
at least one of these edges is present in $G_k \setminus S$.     

\item [Case 2:] There exist $1 \leq i < j \leq k-1$ such that 
$S \cap V_i = \emptyset$ and $S \cap V_j = \emptyset$.\\
It follows by Properties (iii) and (iv) above that $(G_k \setminus S)[V_i \cup V_j]$ 
is connected. Moreover, it follows by Property (v) above that $V_i \cup V_j$ is a dominating
set of $G_k$. Hence, $G_k \setminus S$ is connected in this case.  

\item [Case 3:] There exist $1 \leq i \neq j \leq k-1$ such that 
$S \cap V_i = \emptyset$, $|S \cap V_j| = 2$ and $|S \cap V_t| = 1$ 
for every $t \in [k-1] \setminus \{i,j\}$.\\
It follows by Property (iii) above that $(G_k \setminus S)[V_i]$ is connected. 
Hence, in order to prove that $G_k \setminus S$ is connected, it suffices to 
prove that, for every vertex $u \in V \setminus (V_i \cup S)$ there is a path in 
$G_k \setminus S$ between $u$ and some vertex of $V_i$. Assume first that $u \in V_t$ 
for some $t \in [k-1] \setminus \{i,j\}$. As in Case 1, $(G_k \setminus S)[V_t]$ 
is connected and $E_{G_k \setminus S}(V_t, V_i) \neq \emptyset$. It follows 
that the required path exists. Assume then that $u \in V_j$. If $d_{G_k}(u, V_i) 
> 0$, then there is nothing to prove since $S \cap V_i = \emptyset$. Assume
then that $d_{G_k}(u, V_i) = 0$; it follows by Property (v) above that
$d_{G_k}(u, V_t) > 0$ holds for every $t \in [k-1] \setminus \{i,j\}$.
If $d_{G_k \setminus S}(u, V_t) > 0$ holds for some $t \in [k-1] \setminus \{i,j\}$,
then the required path exists as $(G_k \setminus S)[V_t]$ is connected and,
as previously shown, there is an edge of $G_k \setminus S$ between $V_t$
and $V_i$. Assume then that $d_{G_k \setminus S}(u, V_t) = 0$ holds for every 
$t \in [k-1] \setminus \{i,j\}$. It follows by Property (ii) above
that $d_{G_k}(u, V_j) \geq 3$ and thus $d_{G_k \setminus S}(u, V_j) \geq 1$.
Let $w \in V_j \setminus S$ be a vertex such that $uw \in E_k$. If 
$d_{G_k}(w, V_i) > 0$, then the required path exists. Otherwise,
since $|V_j| \geq 5$ by Property (i) above, it follows by Property 
(vi) above that there exists a vertex $z \in N_{G_k \setminus S}(u, V_j) 
\cup N_{G_k \setminus S}(w, V_j)$ such that $d_{G_k}(z, V_i) > 0$.  
Hence, the required path exists.   
\end{description}
We conclude that $G_k$ is $k$-vertex-connected.
\end{proof}

Note that while ${\mathcal G}_k$ includes very dense graphs, such as $K_n$,
for every $k \geq 3$ and every sufficiently large $n$, this family also 
includes graphs with $\lceil k n/2 \rceil$ edges; that is, $k$-vertex-connected 
graphs which are as sparse as possible. One illustrative example of such a 
graph consists of $k-1$ pairwise vertex disjoint cycles, each of length
$n/(k-1)$ where every pair of cycles is connected by a perfect matching
(in particular, $k-1 \mid n$). The graphs Maker and Red 
will build in the proofs of Theorems~\ref{th::kConMakerBreaker} 
and~\ref{th::kConStrongGame} respectively, are fairly similar to this example.

\section{Auxiliary games}
\label{sec::auxiliary}

In this section we consider several simple games. Some might be interesting
in their own right whereas others are artificial. The results we prove about
these games will be used in our proofs of Theorems~\ref{th::kConMakerBreaker} 
and~\ref{th::kConStrongGame}. We divide this section into several subsections,
each discussing one game.

\subsection{A large matching game}
\label{subsec::largeMatching}

Let  $G = (V_1 \cup V_2, E)$ be a bipartite graph, let $U_1
\subseteq V_1$ and $U_2 \subseteq V_2$, and let $d$ be a positive
integer. The board of the weak game $G(V_1, U_1; V_2, U_2; d)$ is
$E$. Maker wins this game if and only if he accomplishes all of the
following goals:
\begin{description}
\item [(i)] Maker's graph is a matching.
\item [(ii)] $d_M(u) = 1$ for every $u \in (V_1 \setminus U_1) \cup
(V_2 \setminus U_2)$.
\item [(iii)] $d_M(u) = 1$ for every $u \in V_1 \cup V_2$ for which
$d_B(u) \geq d$.
\item [(iv)] $|\{u \in U_1 : d_M(u) = 0\}| \geq |U_1|/2$ and $|\{u \in U_2 :
d_M(u) = 0\}| \geq |U_2|/2$.
\end{description}

\begin{lemma} \label{lem::SpecialMatching}
Let $m$ be a non-negative integer, let $d$ be a positive integer, 
let $8 d^{-1} \leq \varepsilon \leq 0.1$ be a real number and let $n_0 = 
n_0(m, d, \varepsilon)$ be a sufficiently large integer. 
Let $G = (V_1 \cup V_2, E)$ be a bipartite graph which satisfies all 
of the following properties:
\begin{description}
\item [(P1)] $n_0 \leq |V_1| \leq |V_2| \leq (1 + \varepsilon)|V_1|$.
\item [(P2)] $d_G(u, V_2) \geq |V_2| - m$ for every $u \in V_1$.
\item [(P3)] $d_G(u, V_1) \geq |V_1| - m$ for every $u \in V_2$.
\end{description}
Let $U_1 \subseteq V_1$ and $U_2 \subseteq V_2$ be such that
$\varepsilon |V_1| \leq |U_1| \leq 2 \varepsilon |V_1|$ and
$\varepsilon |V_2| \leq |U_2| \leq 2 \varepsilon |V_2|$. Then Maker
(as the first or second player) has a winning strategy for the game
$G(V_1, U_1; V_2, U_2; d)$.
\end{lemma}

\begin{proof}
First we describe a strategy for Maker and then prove it is a
winning strategy. At any point during the game, if Maker is unable
to follow the proposed strategy, then he forfeits the game.

Throughout the game, Maker maintains a matching $M_G$ and a set
$D \subseteq V_1 \cup V_2$ of \emph{dangerous} vertices. A vertex
$v \in V_1 \cup V_2$ is called dangerous if $d_M(v) = 0$ and $d_B(v) 
\geq d$. Initially, $M_G = D = \emptyset$.

For every positive integer $j$, Maker plays his $j$th move as
follows.

\begin{enumerate}[(1)]
\item If $D \neq \emptyset$, then Maker claims an arbitrary free edge $uv \in E$ for
which $u \in D$ and $d_M(v) = 0$. Subsequently, he updates $M_G := M_G \cup \{uv\}$
and $D := D \setminus \{u, v\}$.

\item Otherwise, if there exists a free edge $uv \in E$ such that
$u \in V_1 \setminus U_1$, $v \in V_2 \setminus U_2$ and
$d_M(u) = d_M(v) = 0$, then Maker claims it. Subsequently, he updates
$M_G := M_G \cup \{uv\}$.

\item Otherwise, if there exists a vertex $u \in (V_1 \setminus
U_1) \cup (V_2 \setminus U_2)$ such that $d_M(u) = 0$, then Maker 
claims a free edge $uv \in E$ such that $d_M(v) = 0$. Subsequently, 
he updates $M_G := M_G \cup \{uv\}$.
\end{enumerate}

The game is over as soon as $M_G$ covers $(V_1 \setminus U_1) \cup
(V_2 \setminus U_2)$ and $D = \emptyset$.

It remains to prove that Maker can indeed follow the proposed strategy and
that, by doing so, he wins the game $G(V_1, U_1; V_2, U_2; d)$.

It readily follows from its description that Maker can follow part (2) 
of the proposed strategy. Moreover, it is evident that Maker's 
graph is a matching at any point during the game. Hence, (even if he 
is forced to forfeit the game) he accomplishes goal (i). 
It follows that this game lasts at most $|V_1|$ moves. In particular,
Breaker can create at most $2|V_1|/d \leq \varepsilon |V_1|/4 \leq 
\min \{|U_1|/4, |U_2|/4\}$ dangerous vertices throughout the game. 
Since Maker decreases the size of $D$ whenever he follows part (1) of 
his strategy, we conclude that he follows this part at most 
$\min \{|U_1|/4, |U_2|/4\}$ times. Whenever Maker follows part (3) of the
proposed strategy, $D = \emptyset$ and there is no free edge $uv \in E$ 
such that $u \in V_1 \setminus U_1$, $v \in V_2 \setminus U_2$ and
$d_M(u) = d_M(v) = 0$. It follows by these two conditions and by Properties
(P2) and (P3) that $M_G$ covers at least $|V_1 \setminus U_1| - m - d$ 
of the vertices of $V_1 \setminus U_1$ and at least $|V_2 \setminus U_2|  
- m - d$ of the vertices of $V_2 \setminus U_2$. Since Maker matches
a vertex of $(V_1 \setminus U_1) \cup (V_2 \setminus U_2)$ whenever he 
follows part (3) of the proposed strategy, we conclude that he follows
this part at most $2(m + d) \leq \min \{|U_1|/4, |U_2|/4\}$ times. 
Since, moreover, Maker does not match any vertex of $U_1 \cup U_2$ when 
following part (2), we conclude that he matches at most $\min \{|U_1|/2, |U_2|/2\}$ 
vertices of $U_1 \cup U_2$ throughout the game. It follows that Maker accomplishes 
goal (iv). In particular, Maker can follow part (3) of the proposed 
strategy. Finally, since Maker accomplishes goal (iv), since Breaker creates 
at most $\min \{|U_1|/4, |U_2|/4\}$ dangerous vertices throughout the game,
since Maker plays according to part (1) of the proposed strategy whenever 
$D \neq \emptyset$ and since he decreases $|D|$ whenever he does so, we conclude
that Maker can follows part (1) of the proposed strategy. It now follows that
Maker accomplishes goals (ii) and (iii) as well and thus wins the game.    
\end{proof}

\subsection{A weak positive minimum degree game}
\label{subsec::weakPosMinDeg}

In this subsection we study the weak \emph{positive minimum degree} game
$(E(G), {\mathcal D}_G^1)$, played on the edge set of some given graph $G$.
The family of winning sets ${\mathcal D}_G^1$, consists of the edge sets
of all spanning subgraphs of $G$ with minimum degree at least 1. The following
result was proved in~\cite{HKSS}.  

\begin{theorem} [\cite{HKSS} Corollary 1.3] \label{th::fastPosMinDegKn}
For sufficiently large $n$, Maker has a strategy for winning the weak game
$(E(K_n), {\mathcal D}_{K_n}^1)$ within $\lfloor n/2 \rfloor + 1$ moves.
\end{theorem}

We strengthen Theorem~\ref{th::fastPosMinDegKn} by proving that its
assertion holds even when the board is not complete, though still very 
dense.

\begin{theorem} \label{th::fastPosMinDegDense}
For every positive integer $m$ there exists an integer $n_0 = n_0(m)$ 
such that, for every $n \geq n_0$ and for every graph $G=(V,E)$ on $n$
vertices with minimum degree at least $n-m$, Maker (as the first or
second player) has a strategy for winning the weak positive minimum degree
game $(E(G), {\mathcal D}_G^1)$, within at most $\lfloor n/2 \rfloor + 1$ 
moves.
\end{theorem}

\begin{proof} 
We prove Theorem~\ref{th::fastPosMinDegDense} by induction on $m$. At any
point during the game, let $V_0 := \{u \in V : d_M(u) = 0\}$
denote the set of vertices of $G$ which are isolated in Maker's
graph and let $H := (B \cup (K_n \setminus G))[V_0]$.

In the induction step we will need to assume that $m \geq 3$. Hence,
we first consider the cases $m=1$ and $m=2$ separately. If $m=1$,
then $G = K_n$ and thus the result follows immediately by
Theorem~\ref{th::fastPosMinDegKn}. Assume then that $m=2$ and assume
for convenience that $n$ is even (the proof for odd $n$ is similar
and in fact slightly simpler; we omit the straightforward details).
For every $1 \leq i \leq n/2 - 1$, in his $i$th move, Maker claims 
a free edge $uv$ such that $u,v \in V_0$ and $d_H(u) = \Delta(H)$.
In each of his next two moves, Maker claims a free edge $xy$ such 
that $x \in V_0$ and $y \in V$. 

It is evident that, by following this strategy, Maker wins the 
positive minimum degree game $(E(G), {\mathcal D}_G^1)$, within 
$\lfloor n/2 \rfloor + 1$ moves. It thus remains to prove that
he can indeed follow it. We prove that he can and that $\Delta(H) 
\leq 1$ holds immediately before Breaker's $i$th move for every 
$1 \leq i \leq n/2 - 1$, by induction on $i$. This holds for $i=1$
by assumption. Assume it holds for some $i$. Clearly $\Delta(H) 
\leq 2$ holds immediately after Breaker's $(i+1)$st move. Moreover, 
there are at most two vertices $w \in V_0$ such that
$d_H(w) = 2$ and if there are exactly two such vertices, then they
are connected by an edge of Breaker. In his $(i+1)$st move, Maker 
claims an edge which is incident with a vertex of maximum degree in 
$H$. It follows that $\Delta(H) \leq 1$ holds immediately after
this move. Moreover, since $|V_0| = n - 2i \geq 4$ and $\Delta(H) 
\leq 2$ hold prior to this move, Maker can indeed play his $(i+1)$st 
move according to the proposed strategy. It is clear that Maker can
play his $n/2$th and $(n/2 + 1)$st moves according to the proposed 
strategy.   

Assume then that $m \geq 3$ and that the assertion of the theorem
holds for $m-1$. We present a fast winning strategy for Maker. If at
any point during the game Maker is unable to follow the proposed
strategy, then he forfeits the game. The strategy is divided into
the following two stages.

\noindent \textbf{Stage I}: Maker builds a matching while trying to
decrease $\Delta(H)$. In every move, Maker claims a free edge $uv$
such that $u,v \in V_0$, $d_H(u) = \Delta(H)$ and $d_H(v) = \max
\{d_H(w) : w \in V_0, uw \in E(G \setminus B)\}$. The first stage is
over as soon as $\Delta(H) \leq m-2$ first holds.

\noindent \textbf{Stage II}: Maker builds a spanning subgraph of
$G[V_0]$ with positive minimum degree within $\lfloor |V_0|/2 
\rfloor + 1$ moves.

It is evident that, if Maker can follow the proposed strategy
without forfeiting the game, then he wins the positive minimum degree 
game on $G$ within $\lfloor n/2 \rfloor + 1$ moves. 
It thus suffices to prove that he can indeed do so. First we
prove that Maker can follow Stage I of his strategy, and moreover,
that this stage lasts at most $\frac{(m-1)n}{2m} + 2$ moves. It is
clear from the description of Maker's strategy that the following
property is maintained throughout Stage I.
\begin{description}
\item [$(*)$] $\Delta(H) \leq m$ holds after every move of Breaker.
Moreover, there are at most two vertices $u \in V_0$ such that
$d_H(u) = m$ and if there are exactly two such vertices, then they
are connected by an edge of Breaker.
\end{description}

For every non-negative integer $i$, immediately after Breaker's 
$(i+1)$st move, let $D(i) := \sum_{v \in V_0} d_H(v)$. Note that
$D(i) \geq 0$ for every $i$ and that $D(0) \leq (m-1) n + 2$ (before
the game starts the maximum degree of $H$ is at most $m-1$ and
Breaker claims one edge in his first move). For an arbitrary
non-negative integer $i$, let $uv$ be the edge claimed by Maker in
his $(i+1)$st move. At the time it was claimed, we had $d_H(u) =
\Delta(H) \geq m-1$. Assume first that $d_H(v) \geq 2$ was true as
well. It follows that $D(i+1) \leq D(i) - (m-1) - (m-1) - 2 - 2 + 2
= D(i) - 2m$ (we subtract $2m+2$ from $D(i)$ because of $u,v$ and
their neighbors, and then add $2$ because Breaker claims some edge
in his $(i+2)$nd move). It follows that there can be at most
$\frac{(m-1)n}{2m}$ such moves throughout the first stage. Assume
next that $d_H(v) \leq 1$; note that this entails $d_H(v) \leq m-2$
as $m \geq 3$ by assumption. It follows by Maker's strategy that $u$
is connected by an edge of $H$ to every vertex $x \in V_0$ such that
$d_H(x) \geq 2$. Claiming $uv$ decreases $d_H(w)$ by at least $1$
for every $w \in V_0 \cap N_H(u)$. It follows by Property $(*)$ that
after this move of Maker there is at most one vertex $z \in V_0$
such that $d_H(z) \geq m-1$. It is easy to see that, unless he
forfeits the game, Maker can ensure $\Delta(H) \leq m-2$ in his next
move. It follows that Stage I lasts at most $\frac{(m-1)n}{2m} + 2$
moves as claimed. In particular, we have $|V_0| \geq n/m - 4 > m+1
\geq \Delta(H) + 1$ and thus Maker can indeed follow Stage I of the
proposed strategy without forfeiting the game.

Next, we prove that Maker can follow Stage II of the proposed
strategy. Since the first stage lasts at most $\frac{(m-1)n}{2m} +
2$ moves, $|V_0| \geq n/m - 4 \geq n_0(m-1)$ holds at the beginning
of Stage II. Hence, it follows by the induction hypothesis that 
Maker can win the positive minimum degree game on $(G \setminus B)[V_0]$ 
within $\lfloor |V_0|/2 \rfloor + 1$ moves as claimed. 
\end{proof}

\begin{remark}
The requirement $n/m - 4 \geq n_0(m-1)$ appearing in the the proof of
Theorem~\ref{th::fastPosMinDegDense} shows that the assertion of this
theorem holds even for $m = c \log n/ \log \log n$, where $c > 0$ is
a sufficiently small constant.
\end{remark}

\subsection{A strong positive minimum degree game}
\label{subsec::strongPosMinDeg}

In this subsection we study the strong version of the 
\emph{positive minimum degree} game $(E(G), {\mathcal D}_G^1)$.
We prove the following result.

\begin{theorem} \label{th::strongPosMinDegDense}
For every positive integer $m$ there exists an integer $n_0 = n_0(m)$ 
such that, for every $n \geq n_0$ and for every graph $G=(V,E)$ on $n$
vertices with minimum degree at least $n-m$, Red has a strategy for 
winning the strong positive minimum degree game $(E(G), {\mathcal D}_G^1)$, 
within at most $\lfloor n/2 \rfloor + 1$ moves.
\end{theorem}

\begin{proof} Let ${\mathcal S}_G$ be Maker's strategy for the weak positive
minimum degree game $(E(G), {\mathcal D}_G^1)$ whose existence is guaranteed by
Theorem~\ref{th::fastPosMinDegDense}. If $n$ is odd, then Red simply follows
${\mathcal S}_G$. It follows by Theorem~\ref{th::fastPosMinDegDense} that Red
builds a spanning subgraph of $G$ with positive minimum degree in
$\lfloor n/2 \rfloor + 1$ moves. Since there is no such graph with
strictly less  edges, it follows that Red wins the game. Assume then
that $n$ is even.

We describe a strategy for Red for the strong positive minimum degree game
$(E(G), {\mathcal D}_G^1)$ and then prove it is a winning strategy. If, at any point
during the game, Red is unable to follow the proposed strategy, then
he forfeits the game. At any point during the game, let $V_0 := \{v
\in V : d_R(v) = 0\}$. The strategy is divided into the following
five stages.

\noindent \textbf{Stage I}: In his first move of this stage, Red
claims an arbitrary edge $e_1 = u_1 v_1$. Let $f = x y$ denote the
edge Blue has claimed in his first move; assume without loss of
generality that $x \notin e_1$. Let $A = \{z \in V_0 : xz \notin E\}
\cup \{y\}$. For every $i \geq 2$, immediately before his $i$th move
in this stage, Red checks whether $\Delta(B) \geq 2$, in which case
he skips to Stage V. Otherwise, Red checks whether $A \cap V_0 =
\emptyset$, in which case Stage I is over and Red proceeds to Stage
II. Otherwise, let $w \in A \cap V_0$ be an arbitrary vertex. In his
$i$th move in this stage, Red claims a free edge $w w'$ for some $w'
\in V_0$.

\noindent \textbf{Stage II}: Let $H = (G \setminus B)[V_0 \setminus
\{x\}]$ and let ${\mathcal S}_H$ be the winning strategy for Maker in
the weak positive minimum degree game, played on $E(H)$, which is
described in the proof of Theorem~\ref{th::fastPosMinDegDense}. 
Let $r$ denote the total number of moves Red has played in Stage I. For
every $r < i \leq 3n/8$, immediately before his $i$th move in this
stage, Red checks whether $\Delta(B) \geq 2$, in which case he skips
to Stage V. Otherwise, Red plays his $i$th move according to the
strategy ${\mathcal S}_H$. Once Stage II is over, Red proceeds to
Stage III.

\noindent \textbf{Stage III}: Let $H = (G \setminus B)[V_0 \setminus
\{x\}]$ and let ${\mathcal S}_H$ be the winning strategy for Maker in
the weak positive minimum degree game, played on $E(H)$, which is
described in the proof of Theorem~\ref{th::fastPosMinDegDense}. 
For every $3n/8 < i \leq n/2 - 1$, Red plays his $i$th move according to the
strategy ${\mathcal S}_H$. Once Stage III is over, Red proceeds to
Stage IV.

\noindent \textbf{Stage IV}: Let $z \in V_0 \setminus \{x\}$. If $xz
\in E$ is free, then Red claims it. Otherwise, in his next two
moves, Red claims free edges $x x'$ and $z z'$ for some $x', z' \in
V$. In either case, the game is over.

\noindent \textbf{Stage V}: Let $H = (G \setminus B)[V_0]$ and let
${\mathcal S}_H$ be the winning strategy for Maker in the weak
positive minimum degree game, played on $E(H)$, which is
described in the proof of Theorem~\ref{th::fastPosMinDegDense}. 
In this stage, Red follows ${\mathcal S}_H$ until the end of the game.

We first prove that Red can indeed follow the proposed strategy
without forfeiting the game. We consider each stage separately.

\noindent \textbf{Stage I}: Since $\delta(G) \geq n-m$, it follows
that $|A| \leq m$. Since, moreover, $n$ is sufficiently large with
respect to $m$, we conclude that Red can follow Stage I of the
proposed strategy.

\noindent \textbf{Stage II}: At the beginning of this stage we have
$|V_0 \setminus \{x\}| = n - 2r - 1 \geq 0.99 n $ and $\delta((G
\setminus B)[V_0 \setminus \{x\}]) \geq |V_0| - 1 - m - r \geq |V_0|
- 2m - 2$. Since $n$ is assumed to be sufficiently large with
respect to $m$, it follows by Theorem~\ref{th::fastPosMinDegDense} that the 
required strategy ${\mathcal S}_H$ exists and that Red can indeed 
follow it throughout this stage.

\noindent \textbf{Stage III}: At the beginning of this stage we have
$|V_0 \setminus \{x\}| \geq n/4 - 1$. Moreover, since Red did not
skip to Stage V, it follows that $\delta((G \setminus B)[V_0
\setminus \{x\}]) \geq |V_0| - m - 2$. Since $n$ is assumed to be
sufficiently large with respect to $m$, it follows by Theorem~\ref{th::fastPosMinDegDense} 
that the required strategy ${\mathcal S}_H$ exists and that Red can
indeed follow it throughout this stage.

\noindent \textbf{Stage IV}: If the edge $xz$ is still free, then
Red can clearly claim it. Otherwise, Red can claim a free edge
incident with $x$ and a free edge incident with $z$ since clearly
$\Delta(B) < n/2$.

\noindent \textbf{Stage V}: At the beginning of this stage we have
$|V_0| \geq n/4$. Moreover, since Red has just skipped to Stage V,
it follows that $\delta((G \setminus B)[V_0]) \geq |V_0| - m - 2$.
Since $n$ is assumed to be sufficiently large with respect to $m$, 
it follows by Theorem~\ref{th::fastPosMinDegDense} that the required strategy 
${\mathcal S}_H$ exists and that Red can indeed follow it throughout this
stage.

Next, we prove that if Red follows the proposed strategy, then he
wins the game within at most $n/2 + 1$ moves. If Red reaches Stage V
of the proposed strategy, then the game lasts at most $n/2 + 1$
moves. Since Red reaches Stage V only after Blue wastes a move, it
follows by Theorem~\ref{th::fastPosMinDegDense}that Red wins the game 
in this case. Assume then that Red never reaches Stage V of the proposed 
strategy. It is clear that, at the end of Stage I, Red's graph is a matching. 
Moreover, it follows by the proof of Theorem~\ref{th::fastPosMinDegDense} that
Red's graph is a matching at the end of Stages II and III as well. Moreover,
it is clear that $x \in V_0$ holds at this point. Hence, at the
beginning of Stage IV, we have $V_0 = \{x,z\}$ for some $z \in V$. 
Moreover, by Stage I of the proposed strategy we have $xz \in E$. 
If $xz$ is free, then Red claims it and thus builds a perfect matching 
in $n/2$ moves; hence, he wins the game in this case. Otherwise, the game 
lasts $n/2 + 1$ moves. However, in this case $xz$ was claimed by Blue and thus
$d_B(x) \geq 2$. We conclude that Red wins the game in this case as
well. This concludes the proof of the lemma.
\end{proof}

\section{The Maker-Breaker $k$-vertex-connectivity game}
\label{sec::kConMB}

In this section we prove Theorem~\ref{th::kConMakerBreaker}. 
In our proof we will use the following immediate corollary 
of Theorem 1.1 from~\cite{HS}.

\begin{corollary} \label{cor::HS}
Given a positive integer $n$, let ${\mathcal H}_n^+$ be the family
of all edge sets of Hamilton cycles with a chord of $K_n$. If $n$ 
is sufficiently large, then Maker (as the first or second player)
has a strategy for winning ${\mathcal H}_n^+$ in exactly $n+1$ moves.
\end{corollary}

\textbf{Proof of Theorem~\ref{th::kConMakerBreaker}:}
Assume that $k \geq 4$ (at the end of the proof we will indicate
which small changes have to be made to include the case $k=3$). 
We present a strategy for Maker and then prove it is a winning strategy. 
If at any point during the game Maker is unable to follow the proposed 
strategy, then he forfeits the game. Moreover, if after claiming $k n$ 
edges, Maker has not yet built a $k$-vertex-connected graph, then he 
forfeits the game (we will in fact prove that Maker can build such a graph 
much faster; however, the technical upper bound of $k n$ will suffice for 
the time being). The proposed strategy is divided into the
following four stages.

\textbf{Stage I}: Let $V(K_n) = V_1 \cup V_2 \cup \ldots \cup V_{k-1}$
be an arbitrary equipartition of $V(K_n)$ into $k-1$ pairwise
disjoint sets, that is, $||V_i|-|V_j|| \leq 1$ and $V_i \cap
V_j = \emptyset$ for every $1 \leq i \neq j \leq k-1$. For every $1 \leq
i \leq k-1$, let ${\mathcal S}_i$ be a winning strategy for Maker in
the game ${\mathcal H}_{|V_i|}^+$ played on $E(K_n[V_i])$ whose
existence is ensured by Corollary~\ref{cor::HS}. In this stage,
Maker's goal is to build a Hamilton cycle of $K_n[V_i]$ with a chord
for every $1 \leq i \leq k-1$ while limiting the degree of certain
vertices in Breaker's graph. If Maker is unable to accomplish both goals 
within $2n$ moves, then he forfeits the game. For every vertex $v \in V(K_n)$,
let $1 \leq i_v \leq k-1$ be the (unique) index such that $v\in
V_{i_v}$. Throughout this stage, Maker maintains a set $D \subseteq
V(K_n)\times [k-1]$ of \emph{dangerous} pairs. A pair $(v,i) \in
V(K_n)\times [k-1]$ is called \emph{dangerous} if $v \notin V_i$,
$d_B(v,V_i) \geq 0.9|V_{i}|$, $d_M(v,V_i) = 0$ and $d_M(v) < k$. 
Initially, $D = \emptyset$. For every positive integer $j$, let 
$e_j = u v$ denote the edge which has been claimed by Breaker in 
his $j$th move. Maker plays his $j$th move as follows.

\begin{enumerate}[(i)]
\item If $e_{j}\in E(V_i)$ for some $1\leq i \leq k-1$ and $M[V_i]$ is
not yet a Hamilton cycle (of $K_n[V_i]$) with a chord, then Maker
responds in this board according to the strategy ${\mathcal S}_i$.

\item Otherwise, if $D\neq \emptyset$, let $(z,i)\in D$ be
a dangerous pair such that $d_B(z,V_i)=\max
\{d_B(w,V_{\ell}):(w,\ell)\in D\}$. Maker claims a free edge $zw$
such that $w\in V_i$ and $d_M(w,V_{i_z})=0$. Subsequently, Maker
updates $D:=D\setminus \{(z,V_i),(w,V_{i_z})\}$.

\item Otherwise, if there exists $x\in \{u,v\}$ such that $M[V_{i_x}]$ 
is not yet a Hamilton cycle with a chord, then Maker plays as follows. 
Let $y\in \{u,v\}$ be such that $d_B(y,V(K_n) \setminus V_{i_y}) = 
\max \{d_B(v,V(K_n)\setminus V_{i_v}), d_B(u,V(K_n) \setminus V_{i_u})\}$
and let $z \in \{u,v\} \setminus \{y\}$. If $M[V_{i_y}]$ is not yet a 
Hamilton cycle with a chord, then Maker follows ${\mathcal S}_{i_y}$ on
the board $E(V_{i_y})$, otherwise he follows ${\mathcal S}_{i_z}$ on
$E(V_{i_z})$.   

\item Otherwise, Maker plays according to ${\mathcal S}_i$ in a
board $E(V_i)$ for some $1\leq i \leq k-1$ such that $M[V_i]$ is not
yet a Hamilton cycle with a chord.
\end{enumerate}
As soon as $M[V_i]$ is a Hamilton cycle with a chord
for every $1 \leq i \leq k-1$ and $D = \emptyset$, this stage is
over and Maker proceeds to Stage II.

\textbf{Stage II}: Let $C$ be the set of endpoints of the chords
of $\bigcup_{i=1}^{k-1} M[V_i]$. At any point during this stage, let  
$Y_C := \{v \in C: d_M(v) < k\}$, let $Y_D := \{v \in V(K_n) : d_M(v) < k 
\textrm{ and } d_B(v) \geq k^{10}\}$ and let $ Y := Y_C \cup Y_D$. 
For as long as $Y \neq \emptyset$, Maker picks an
arbitrary vertex $v \in Y$ and plays as follows. Let $t = d_M(v)$ and
let $\{i_1, \ldots ,i_{k-t}\} \subseteq [k-1] \setminus \{i_v\}$ be
$k-t$ distinct indices such that $d_M(v, V_{i_j}) = 0$ for every $1 \leq
j \leq k-t$. In his next $k-t$ moves, Maker claims $k-t$ free edges
$\{v v_{i_j} : 1 \leq j \leq k-t\}$ such that $v_{i_j} \in V_{i_j}$ and
$d_M(v_{i_j}, V_{i_v}) = 0$ for every $1 \leq j \leq k-t$.

As soon as $Y = \emptyset$, this stage is over and Maker proceeds to
Stage III.

\textbf{Stage III}: For every $1 \leq i \neq j \leq k-1$, let
$A_{ij} \subseteq V_i$ denote the set of vertices $v \in V_i$ such
that $d_M(v) < k$ and $d_M(v,V_j) = 0$. Moreover, for every $1 \leq i \neq
j \leq k-1$, let $B_{ij} \subseteq A_{ij}$ be sets which satisfy all
of the following properties:
\begin{enumerate}[(P1)]
\item $B_{ij} \cap B_{i\ell} = \emptyset$ for every $1 \leq i \leq k-1$
and for every $1\leq j\neq \ell\leq k-1$.

\item $n/k^6 \leq |B_{ij}| \leq 2n/k^6$ for every $1 \leq i \neq j \leq
k-1$.

\item $dist_{M[V_i]}(u,v)\geq 2$ for every $1\leq i \leq k-1$ and for every two distinct
vertices $u,v\in \bigcup _{j\in [k-1]\setminus \{i\}} B_{ij}$.
\end{enumerate}

For every $1 \leq  i < j \leq k-1$ let $G_{ij} = (A_{ij} \cup A_{ji},
E_{K_n \setminus B}(A_{ij},A_{ji}))$ and let ${\mathcal S}_{ij}$ be the
winning strategy for Maker in the game
$G_{ij}(A_{ij}, B_{ij}; A_{ji}, B_{ji}; 2 k^{10})$ which is described in
the proof of Lemma~\ref{lem::SpecialMatching}.

At any point during this stage, for every $1 \leq i < j \leq k-1$,
Maker maintains a matching $M_{ij}$ of the board $E(G_{ij})$ and a
set $D\subseteq V(K_n)$ of \emph{dangerous} vertices. A vertex $v\in
V(K_n)$ is called dangerous if $v\in B_{ij}$ for some $1 \leq i \neq j
\leq k-1$ (without loss of generality assume $i<j$) and, moreover,
$v$ satisfies all of the following properties:
\begin{enumerate}[(1)]
\item $v$ is not matched in $M_{ij}$.
\item $M_{ij}$ covers $(A_{ij} \setminus B_{ij}) \cup (A_{ji} \setminus B_{ji})$.
\item $d_B(v) \geq k^{10}$.
\end{enumerate}

Initially, $D = M_{ij} = \emptyset$ for every $1 \leq i < j \leq k-1$.

Let $r$ denote the number of moves Maker has played throughout Stages
I and II. For every $s > r$, let $e_s$ denote the edge that has been
claimed by Breaker in his $s$th move. Maker plays his $s$th move as
follows:

\begin{enumerate} [(i)]

\item If $e_s \in E(G_{ij})$ for some $1 \leq i < j \leq k-1$ and $M_{ij}$
does not yet cover $\left(A_{ij}\setminus B_{ij}\right)\cup
\left(A_{ji}\setminus B_{ji}\right)$, then Maker responds in the board 
$E(A_{ij},A_{ji})$ according to the strategy ${\mathcal S}_{ij}$.

\item Otherwise, if $D\neq \emptyset$, then Maker claims a free edge $uv$ 
between two sets $B_{ij}$ and $B_{ji}$ such that the following properties hold.
\begin{enumerate} [(a)]
\item $u \in D$.
\item $d_B(u) = \max \{d_B(w) : w \in D\}$.
\item $M_{ij}$ covers $(A_{ij}\setminus B_{ij})\cup
(A_{ji}\setminus B_{ji})$.
\end{enumerate}
Maker updates $D:=D\setminus \{u,v\}$.

\item Otherwise, Maker picks arbitrarily $1\leq i < j \leq k-1$
such that $M_{ij}$ does not yet cover $(A_{ij}\setminus B_{ij})\cup
(A_{ji}\setminus B_{ji})$ and plays in the board $E(A_{ij},A_{ji})$
according to the strategy ${\mathcal S}_{ij}$.
\end{enumerate}

As soon as $M_{ij}$ covers $\left(A_{ij}\setminus B_{ij}\right)\cup
\left(A_{ji}\setminus B_{ji}\right)$ for every $1\leq i < j \leq
k-1$ and $D=\emptyset$, this stage is over and Maker proceeds to
Stage IV.

\textbf{Stage IV}: Let $U = \{v \in V(K_n) : d_M(v) = k-1\}$ and let 
$H := (K_n \setminus B)[U]$. Let ${\mathcal S}_H$ be a strategy
for Maker for winning the positive minimum degree game $(E(H), {\mathcal D}_H^1)$
within $\lfloor |U|/2 \rfloor + 1$ moves. In this stage
Maker follows ${\mathcal S}_H$ until $\delta(M) \geq k$ first occurs; 
at this point the game is over.

It is evident that if Maker can follow the proposed strategy without
forfeiting the game, then, by the end of the game, he builds a 
graph $M \in {\mathcal G}_k$, which is $k$-vertex-connected by
Proposition~\ref{prop::kconnected}. It thus suffices to prove that 
Maker can indeed follow the proposed strategy without forfeiting the 
game and that, by doing so, he builds an element of ${\mathcal G}_k$ 
within $\lfloor k n/2 \rfloor + 1$ moves.

Our first goal is to prove that Maker can indeed follow the proposed
strategy without forfeiting the game. We consider each stage
separately.

\textbf{Stage I}: Since $|V_i| \geq \lfloor n/(k-1) \rfloor$ for every
$1 \leq i \leq k-1$ and since $n$ is sufficiently large with respect to $k$, 
it follows by Corollary~\ref{cor::HS} that Maker can follow part (i) of
the proposed strategy for this stage.

Recall that, by definition, this stage lasts at most $2n$ moves and
that $d_B(v, V_i) \geq 0.9 |V_i| \geq 0.9n/k$ holds for every dangerous pair 
$(v,i) \in D$. Therefore, throughout Stage I, Breaker can create at most
$4n/ \left(\frac{0.9n}{k} \right) \leq 5k$ such pairs. We claim that 
at any point during Stage I, $d_B(v, V_i) \leq 0.95 |V_i|$ holds for 
every vertex $v \in V(K_n)$ and every $i \in [k-1] \setminus \{i_v\}$.
This is immediate by the definition of $D$ for every pair $(v,i) \in 
(V(K_n) \times [k-1]) \setminus D$. Consider a point during this stage where 
$D \neq \emptyset$ (if this never happens, then there is nothing left to prove). 
If Breaker plays in $\bigcup_{i=1}^{k-1} E(V_i)$, then he does not
increase $d_B(v, V_i)$ for any pair $(v,i) \in D$. Otherwise, Maker follows 
part (ii) of the proposed strategy for this stage and thus decreases the 
size of $D$. It follows that, throughout Stage I, Maker follows part (ii)
of the proposed strategy at most $5k$ times. Since $n$ is
sufficiently large with respect to $k$, it follows that, throughout 
Stage I, $d_B(v, V_i) \leq 0.9 |V_i| + 5k \leq 0.95 |V_i|$ holds for 
every $v \in V(K_n)$ and every $i \in [k-1] \setminus \{i_v\}$ as claimed.  
Since Maker follows part (ii) of the proposed strategy at most $5k$ times
and since he only claims edges of $\bigcup_{i=1}^{k-1} E(V_i)$ when following
parts (i), (iii) or (iv) of the strategy, it follows that, throughout 
Stage I, $|\{u \in V_i : d_M(u, V_j) = 0\}| \geq 0.99 |V_i|$ holds for every 
$1 \leq i \neq j \leq k-1$. Hence, Maker can follow part (ii) of the proposed 
strategy for this stage without forfeiting the game.     

Finally, it readily follows from Corollary~\ref{cor::HS} that Maker can follow 
parts (iii) and (iv) of the proposed strategy for this stage.

It thus suffices to prove that Maker can achieve his goals for this
stage within at most $2n$ moves. This readily follows from the following
three simple observations.

\begin{enumerate} [(a)]
\item According to Corollary~\ref{cor::HS}, for every $1 \leq i \leq k-1$,
Maker can build a Hamilton cycle of $K_n[V_i]$ with a chord in
$|V_i| + 1$ moves.

\item Whenever Maker follows parts (i), (iii) or (iv) of the proposed
strategy for this stage, he plays according to ${\mathcal S}_i$
for some $1 \leq i \leq k-1$.

\item As previously noted, Maker follows part (ii) 
of the proposed strategy at most $5k$ times.
\end{enumerate}

It follows that Stage I lasts at most
$\sum_{i=1}^{k-1} \left(|V_i|+1 \right) + 5k = n + (k-1) + 5k < 2n$ moves.

We conclude that Maker can follow the proposed strategy for this stage,
including the time limits it sets, without forfeiting the game.

\textbf{Stage II}: Since the entire game lasts at most $k n$ moves, 
it follows that $|\{u \in V(K_n) : d_B(u) \geq k^{10}\}| \leq 
2 k n/ k^{10}$ holds at any point during the game. Hence, 
$|Y| \leq 2(k-1) + 2 n/ k^9 \leq 3 n/ k^9$ holds at any point during
this stage. Since $D = \emptyset$ at the end of Stage I and since 
Maker spends at most $k$ moves on every vertex of $Y$, it follows 
that, at any point during this stage, $d_B(v,V_i) \leq 0.9 |V_i| + 
3 n/ k^8 \leq 0.95 |V_i|$ holds for every vertex $v \in Y$ and for 
every $i \in [k-1] \setminus \{i_v\}$. Since, as noted above, 
$|\{u \in V_i : d_M(u, V_j) = 0\}| \geq 0.99 |V_i|$ holds for every 
$1 \leq i \neq j \leq k-1$ at the end of Stage I, it follows that 
$|\{u \in V_i : d_M(u, V_j) = 0\}| \geq 0.98 |V_i|$ holds for every 
$1 \leq i \neq j \leq k-1$ throughout Stage II. We conclude that Maker 
can follow the proposed strategy for this stage without forfeiting the 
game.

\textbf{Stage III}: For every $1 \leq i \leq k-1$, let $A_i := \{u \in
V_i : d_M(u) = 2\}$. Since Maker follows Stages I and II of the proposed
strategy, we conclude that $|A_i| \geq \lfloor n/(k-1) \rfloor - (k+1)
(5k + 2(k-1) + 2 n/ k^9) \geq 0.9 n/k$ holds for every such $i$.
Moreover, since Stage II lasts at most $k |Y| \leq n/k^7$ moves, it follows
that $||A_{ij}| - |A_{ji}|| \leq n/k^7$ holds for every $1 \leq i < j \leq k-1$.

For every $1 \leq i \leq k-1$, let $B_i \subseteq A_i$ be a set which 
satisfies $|B_i| \geq \lfloor |A_i|/2 \rfloor \geq |A_i|/3$
and $dist_{M[V_i]}(u,v) \geq 2$ for every $u,v \in B_i$ (one example
of such a set is obtained by enumerating the elements of $A_i$ according 
to their order of appearance on the Hamilton cycle of $K_n[V_i]$ and
taking either all even indexed vertices or all odd indexed vertices). 
Let $B_i = B_i^{(1)} \cup \ldots \cup B_i^{(i-1)} \cup B_i^{(i+1)} \cup 
\ldots \cup B_i^{(k-1)}$ be an equipartition of $B_i$. For every $1 \leq i <
j \leq k-1$ let $B_{ij} \subseteq B_i^{(j)}$ and $B_{ji} \subseteq
B_j^{(i)}$ be chosen such that Property (P2) in the description of
the proposed strategy for this stage holds. Note that Properties (P1)
and (P3) hold as well by the construction of the $B_i$'s and the 
$B_i^{(j)}$'s.
 
Since, as noted above, $||A_{ij}| - |A_{ji}|| \leq n/k^7$ holds for 
every $1 \leq i < j \leq k-1$, since $d_B(u) < k^{10}$ holds for 
every $u \in A_i$ by Stage II of the proposed strategy and since
$n$ is sufficiently large with respect to $k$, it follows by
Lemma~\ref{lem::SpecialMatching} (with $\varepsilon = k^{-4}$)
that Maker can follow parts (i) and (iii) of the proposed strategy 
for this stage.

Moreover, since $d_B(v) \geq k^{10}$ holds for every dangerous
vertex and since the entire game lasts at most $k n$ moves, it follows
that Breaker can create at most $2 k n /k^{10} \leq n/k^{8}$ such
vertices. Since Maker spends exactly one move to treat a dangerous
vertex and since $|B_{ij}| \geq n/k^6$ holds by construction for every 
$1 \leq i \neq j \leq k-1$, it follows that Maker can indeed 
follow part (ii) of the proposed strategy for this stage.

\textbf{Stage IV}: Whenever Maker follows part (ii) of the
proposed strategy for this stage, he increases the degrees
of two vertices by 1 each and decreases the size of $D$. Since 
the entire game lasts at most $k n$ moves and since $d_B(v)
\geq k^{10}$ holds for every $v \in D$, it follows that Maker
follows part (ii) of the strategy at most $2 n/k^9$ times. 
It follows by Lemma~\ref{lem::SpecialMatching} and by Property 
(P2) that  
$$
|U| \geq \sum_{1 \leq i \neq j \leq k-1} |B_{ij}|/2 - 4 n/k^9 \geq 
\binom{k-1}{2} n/(2k^6) - 4 n/k^9 \geq n/(3k^6) \,.
$$ 
Since $n$ is sufficiently large with respect to $k$, it thus follows 
by Theorem~\ref{th::fastPosMinDegDense} that Maker can follow the strategy 
${\mathcal S}_H$ throughout this stage without forfeiting the game. 

It remains to prove that, by following the proposed strategy, Maker wins 
the game within $\lfloor k n/2 \rfloor + 1$ moves. It follows by
Theorem~\ref{th::fastPosMinDegDense} that Stage IV lasts at most
$\lfloor |U|/2 \rfloor + 1$ moves. It thus suffices to prove that 
$\delta(M) \leq k$ holds throughout Stages I, II and III. This follows
quite easily from the description of Maker's strategy. There is one 
exception though. If $d_M(u, V_i) > 0$ for every $i \in [k-1] \setminus 
\{i_u\}$ and only then $u$ becomes an endpoint of a chord, then we have 
$d_M(u) = k+1$. In order to overcome this problem, we include part (iii)
of the strategy for Stage I. Recall that Maker follows part (ii) of the 
proposed strategy for this stage at most $5k$ times and that $||V_i| - |V_j||
\leq 1$ holds for every $1 \leq i < j \leq k-1$. It thus follows by 
part (iii) of the proposed strategy that if $d_B(u, V_i) \geq 0.9 |V_i|$ and $d_B(u, V_j) 
\geq 0.9 |V_j|$ hold for two distinct indices $i, j \in [k-1] \setminus 
\{i_u\}$, then $M[V_{i_u}]$ is already a Hamilton cycle with a chord;
in particular we know whether $u$ is an endpoint of this chord or not.
Since $k \geq 4$, we can afford to wait until a vertex appears in two
dangerous pairs. For $k=3$ we have no choice but to ensure that if a 
vertex $u$ satisfies $d_M(u, V_i)$ for $i \neq i_u$, then it will not 
become an endpoint of the chord of $M[V_{i_u}]$. In order to ensure 
this, one has to slightly alter Maker's strategy for the game 
${\mathcal H}_{|V_{i_u}|}^+$. This can be done by adjusting the strategy
given in the proof of Theorem 1.1 in~\cite{HS} or the strategy
given in the proof of Theorem 1.1 in~\cite{HKSS} (the latter is easier).
Note that this solution works for every $k \geq 3$. However, where possible,
we preferred a solution which uses Maker's strategy for the Hamilton cycle 
with a chord game as a black box.

This concludes the proof of the theorem.
{\hfill $\Box$ \medskip\\}

\section{The strong $k$-vertex-connectivity game}
\label{sec::strongCon}

\textbf{Proof of Theorem~\ref{th::kConStrongGame}:}
Let $k \geq 3$ be an integer and assume first that $k n$ is odd. Red
simply follows Maker's strategy for the weak $k$-vertex-connectivity
game $(E(K_n), {\mathcal C}_n^k)$ whose existence is guaranteed by 
Theorem~\ref{th::kConMakerBreaker}. It follows by Theorem~\ref{th::kConMakerBreaker} 
that he builds a $k$-vertex-connected graph in $\lfloor k n/2 \rfloor + 1$ moves.
Since, for odd $k n$, there is no graph $G$ on $n$ vertices such that
$\delta(G) \geq k$ and $e(G) \leq \lfloor k n/2 \rfloor$, it follows that Red 
wins the strong $k$-vertex-connectivity game $(E(K_n), {\mathcal C}_n^k)$.

Assume then that $k n$ is even. First, we present a strategy for Red and 
then prove it is a winning strategy. If at any point during the game Red is 
unable to follow the proposed strategy, then he forfeits the game. 
The proposed strategy is divided into the following two stages.

\textbf{Stage I}: Let ${\mathcal S}_M$ be the winning strategy for
Maker in the weak game $(E(K_n), {\mathcal C}_n^k)$ which is described 
in the proof of Theorem~\ref{th::kConMakerBreaker}. In this stage, Red 
follows Stages I, II and III of the strategy ${\mathcal S}_M$. As soon 
as Red first reaches Stage IV of ${\mathcal S}_M$, this stage is over 
and Red proceeds to Stage II.

\textbf{Stage II}: Let $U_0 := \{v \in V(K_n) : d_R(v) = k-1\}$ and 
let $G=(K_n\setminus B)[U_0]$. Let ${\mathcal S}_G$ be the winning
strategy for Red in the strong positive minimum degree game $(E(G), 
{\mathcal D}_G^1)$ which is described in the proof of Theorem~\ref{th::strongPosMinDegDense}. 
We distinguish between the following three cases.
\begin{enumerate}[(1)]
\item If $\Delta(B) > k$, then Red continues playing according to the
strategy ${\mathcal S}_M$ until the end of the game. That is, he follows
Stage IV of ${\mathcal S}_M$ until his graph first becomes $k$-vertex-connected.

\item Otherwise, if $d_B(v) \leq k-1$ for every $v \in U_0$, then Red
plays the strong positive minimum degree game $(E(G), {\mathcal D}_G^1)$ 
according to the strategy ${\mathcal S}_G$ until his graph becomes 
$k$-vertex-connected.

\item Otherwise, let $x \in U_0$ be a vertex such that $d_B(x) = k$. Let
$H = G \setminus \{x\}$ and let ${\mathcal S}_H$ be the winning strategy
for Red in the strong positive minimum degree game $(E(H), {\mathcal D}_H^1)$ 
which is described in the proof of Theorem~\ref{th::strongPosMinDegDense}. 
Let $r$ denote the total number of moves Red has played so far. 
This case is further divided into the following four substages.
\begin{enumerate}[(i)]
\item For every $r < i \leq k n/2 - |U_0|/3$, immediately before 
his $i$th move in this stage, Red checks whether $\Delta(B) > k$, 
in which case he skips to Substage (iv). Otherwise, Red plays his 
$i$th move according to the strategy ${\mathcal S}_H$. As soon as
this substage is over Red proceeds to Substage (ii).

\item For every $k n/2 - |U_0|/3 < i \leq k n/2 - 1$, Red plays 
his $i$th move according to the strategy ${\mathcal S}_H$. 
When this substage is over Red proceeds to Substage (iii).

\item Let $z \in U_0 \setminus \{x\}$ be a vertex of degree $k-1$
in Red's graph. If the edge $xz \in E(K_n)$ is free, then Red 
claims it. Otherwise, in his next two moves, Red
claims free edges $x x'$ and $z z'$ for some $x', z' \in V(K_n)$. In
both cases the game is over.

\item Let $U := \{v \in V(K_n) : d_R(v) = k-1\}$ and let 
$G' = (K_n \setminus B)[U]$. Let ${\mathcal S}_{G'}$ be the winning strategy
for Red in the strong positive minimum degree game $(E(G'), {\mathcal D}_{G'}^1)$ 
which is described in the proof of Theorem~\ref{th::strongPosMinDegDense}. 
In this substage, Red follows ${\mathcal S}_{G'}$ until the end, that is, 
until his graph first becomes $k$-vertex-connected.
\end{enumerate}
\end{enumerate}

It is evident that if Red can follow the proposed strategy without
forfeiting the game, then, by the end of the game, he builds a 
graph $R \in {\mathcal G}_k$, which is $k$-vertex-connected by
Proposition~\ref{prop::kconnected}. It thus suffices to prove that 
Red can indeed follow the proposed strategy without forfeiting the 
game, that he builds an element of ${\mathcal G}_k$ within 
$\lfloor k n/2 \rfloor + 1$ moves and that he does so before 
$\delta(B) \geq k$ first occurs.

Our first goal is to prove that Red can indeed follow the proposed
strategy without forfeiting the game. We consider each stage
separately.

\textbf{Stage I}: Since $n$ is sufficiently large with respect to
$k$, it follows by Theorem~\ref{th::kConMakerBreaker} that Red can
follow Stage I of the proposed strategy.

\textbf{Stage II}: We consider each of the three cases separately.
\begin{enumerate}[(1)]
\item Since Red has played all of his moves in Stage I according 
to the strategy ${\mathcal S}_M$, it follows by Theorem~\ref{th::kConMakerBreaker} 
that he can continue doing so until the end of the game.

\item Since Red has played all of his moves in Stage I according 
to the strategy ${\mathcal S}_M$, it follows by the proof of 
Theorem~\ref{th::kConMakerBreaker} that $|U_0| = \Omega(n)$ holds
at the beginning of Stage II. Since we are not in Case (1), it
follows that $\delta(G) \geq |U_0| - k$. Since, moreover, $n$
is sufficiently large with respect to $k$, it follows by 
Theorem~\ref{th::strongPosMinDegDense} that Red can indeed
follow the proposed strategy for this case without forfeiting the
game.

\item As previously noted, $|U_0| = \Omega(n)$ holds at the beginning 
of Stage II. Since we are not in Case (1), it follows that $\delta(H) 
\geq |U_0| - 1 - k$. Since, moreover, $n$ is sufficiently large with 
respect to $k$, it follows by Theorem~\ref{th::strongPosMinDegDense} 
that Red can follow Substage (i) of the proposed strategy for this case. 
Since $\Delta(B) \leq k$ holds at the beginning of Substage (ii) 
(otherwise Red would have skipped to Substage (iv)), it follows by an 
analogous argument that Red can follow Substage (ii) of the proposed 
strategy for this case as well. It follows by Substages (i) and (ii)
of the proposed strategy that, at the beginning of Substage (iii),
there are exactly two vertices of degree $k-1$ in Red's graph, one
of which is $x$. Denote the other one by $z$. Since $\Delta(B) \leq k$ 
holds at the beginning of Substage (ii) and since this entire substage 
clearly lasts at most $|U_0|/3$ moves, it follows that $d_B(x) \leq 
k + |U_0|/3 < n/2$ and $d_B(z) \leq k + |U_0|/3 < n/2$ hold at the beginning of
Substage (iii). Hence, Red can follow Substage (iii) of the proposed 
strategy for this case. Finally, since $\Delta(B) \leq k+1$ and $|U| 
= \Omega(n)$ clearly hold at the beginning of Substage (iv) and since
$n$ is sufficiently large with respect to $k$, it follows by 
Theorem~\ref{th::strongPosMinDegDense} that Red can follow Substage 
(iv) of the proposed strategy for this case.      
\end{enumerate}

It is evident from the description of the proposed strategy that
the game lasts at most $k n/2 + 1$ moves. Hence, in order to complete 
the prove of the theorem, it suffices to show that, if the game
lasts exactly $k n/2 + 1$ moves, then $\Delta(B) > k$. This clearly
holds if the game ends in Case (1) or in Substage (iv) of Case (3).
If the game ends in Case (2), then this follows by 
Theorem~\ref{th::strongPosMinDegDense}. Finally, if the game lasts 
exactly $k n/2 + 1$ moves and ends in Substage (iii) of Case (3),
then $d_B(x) \geq k+1$ must hold.
 
This concludes the proof of the theorem.
{\hfill $\Box$\medskip\\}

\section{Concluding remarks and open problems}
\label{sec::openprob}

\begin{description}
\item [A more natural fastest possible strategy for the minimum-degree-$k$ game.] As noted
in Corollary~\ref{cor::MinDegk} (respectively Corollary~\ref{cor::StrongMinDegk}), Maker 
(respectively Red) can win the weak (respectively strong) minimum-degree-$k$ game 
$(E(K_n), {\mathcal D}_n^k)$ within $\lfloor k n/2 \rfloor + 1$ moves by following his 
strategy for the weak (respectively strong) $k$-vertex-connectivity game $(E(K_n), {\mathcal C}_n^k)$.
While useful, this is not a very natural way to play this game. We have found
a much more natural strategy for Maker (respectively Red) to win the weak (respectively strong) game 
$(E(K_n), {\mathcal D}_n^k)$ within $\lfloor k n/2 \rfloor + 1$ moves. It consists
of two main stages. In the first stage, Maker (respectively Red) builds a graph
with minimum degree $k-1$ and maximum degree $k$. This is done almost arbitrarily
except that Maker (respectively Red) ensures that, if a vertex has degree $k-1$
in his graph, then its degree in Breaker's (respectively Blue's) graph will not
be too large. In the second stage, he plays the weak (respectively strong) positive 
minimum degree game $(E(K_n), {\mathcal D}_n^1)$ on the graph induced by the
vertices of degree $k-1$ in his graph. We omit the details.      

\item [Explicit winning strategies for other strong games.] Following the observation
made in~\cite{FH} that fast winning strategies for Maker in a weak game 
have the potential of being upgraded to winning strategies for Red in the
corresponding strong game, we have devised a winning strategy for Red in the 
strong $k$-vertex-connectivity game. It is plausible that one could devise
a winning strategy for other strong games, where a fast strategy 
is known for the corresponding weak game. One natural candidate is the \emph{specific 
spanning tree} game. This game is played on the edge set of $K_n$ for some sufficiently
large integer $n$. Given a tree $T$ on $n$ vertices, the family of winning sets ${\mathcal T}_n$
consists of all copies of $T$ in $K_n$. It was proved in~\cite{FHK} that Maker
has a strategy to win the weak game $(E(K_n), {\mathcal T}_n)$ within $n + o(n)$ moves
provided that $\Delta(T)$ is not too large.   

On the other hand, there are weak games for which Maker has a winning strategy 
and yet Breaker can refrain from losing quickly. Consider for example the 
\emph{Clique} game $RG(n,q)$. The board of this game is the edge set of $K_n$ and
the family of winning sets consists of all copies of $K_q$ in $K_n$. It is 
easy to see that for every positive integer $q$ there exists an integer $n_0$
such that Maker (respectively Red) has a strategy to win the weak (respectively
strong) game $RG(n,q)$ for every $n \geq n_0$.
However, it was proved in~\cite{BeckFast} that Breaker can refrain from 
losing this game within $2^{q/2}$ moves. The current best upper bound on
the number of moves needed for Maker in order to win $RG(n,q)$ is $2^{2q/3} \cdot f(q)$,
where $f(q)$ is some polynomial in $q$ (see~\cite{Gebauer}). Note that this
upper bound does not depend on the size of the board, in particular, it holds
for an infinite board as well. Given that an exponential lower bound on the
number of moves is known, it would be very 
interesting to find an explicit winning strategy for Red in the strong game $RG(n,q)$
for every positive integer $q$ and sufficiently large $n$.
Moreover, it would be interesting to determine whether Red can win this game
on an infinite board.   
\end{description}


\begin{thebibliography}{99}

\bibitem{BeckFast}
J. Beck, Ramsey Games, {\em Discrete Mathematics} 249 (2002), 3-–30.

\bibitem{FH}
A. Ferber and D. Hefetz, Winning strong games through fast strategies for weak games,
{\em The Electronic Journal of Combinatorics} 18(1) (2011), P144.

\bibitem{FHK}
A. Ferber, D. Hefetz and M. Krivelevich, Fast embedding of spanning trees in biased Maker-Breaker games,
{\em European Journal of Combinatorics}, 33 (2012), 1086-–1099. 

\bibitem{Gebauer}
H. Gebauer, On the clique-game, {\em European Journal of Combinatorics} 33(1) (2012), 8--19.

\bibitem{HKSS}
D. Hefetz, M. Krivelevich, M. Stojakovi\'c and T. Szab\'o, Fast
winning strategies in Maker-Breaker games, {\em Journal of
Combinatorial Theory, Ser. B.} 99 (2009), 39--47.

\bibitem{HS}
D. Hefetz and S. Stich, On two problems regarding the Hamilton cycle game,
{\em The Electronic Journal of Combinatorics} 16(1) (2009), R28.

\bibitem{Lehman}
A. Lehman, A solution of the Shannon switching game, {\em J. Soc. Indust. Appl.
Math.} 12 (1964), 687--725.

\bibitem{West}
D.\ B.\ West, {\bf Introduction to Graph Theory}, Prentice Hall,
2001.

\end{thebibliography}
\end{document}